\documentclass[10pt]{paper}%
\usepackage[T1]{fontenc}
\usepackage{enumerate}
\usepackage{amsmath,amssymb}
\usepackage{fullpage}
\usepackage{amsthm}
\usepackage{amsmath}
\usepackage{graphicx}
  \usepackage[all]{xy}
\usepackage{amsfonts}
\usepackage{amssymb}%

\newcommand{\bs}{{\boldsymbol{s}}}

\newcommand{\maps}{\operatorname{Maps}}
\newtheorem{Th}{Theorem}[section]
\newtheorem{lemma}[Th]{Lemma}
\newtheorem{assump}[Th]{Assumption}

\newtheorem{Prop}[Th]{Proposition}

\theoremstyle{remark}
\newtheorem{Rem}[Th]{Remark}{\rmfamily}
\theoremstyle{definition}
{\rmfamily}
{\rmfamily}

\begin{document}

\title{On quantized  decomposition maps for graded algebras}
\author{Maria Chlouveraki and Nicolas Jacon}
\maketitle
\date{}

\begin{abstract}
In this note, we study a factorization result for graded decomposition maps associated with the specializations of graded  algebras.
 We obtain  results previously known  only in the ungraded setting.  
\end{abstract}

\section{Introduction}

Using the works of Khovanov and Lauda \cite{KL} and Rouquier \cite{Rr}, Brundan and Kleshchev   have  shown in \cite{BK1}  the existence of a $\mathbb{Z}$-grading on the 
 Hecke algebras associated with the complex reflection groups of type $G(l,1,n)$ over any field $F$. The next natural step is to study the graded representation theory  for these algebras. In the characteristic $0$ case,  using  a grading on the Specht modules explicitly constructed in \cite{BKW},  Brundan and Kleshchev have shown in \cite{BK1,BK2}  the existence of a  certain ``graded decomposition matrix''  (with coefficients in $\mathbb{N}[q,q^{-1}]$ for some {indeterminate $q$).} In addition, they have proved that  this matrix, which  may be viewed as a quantization of the usual decomposition matrix,  corresponds to the canonical basis matrix  in one realization of the irreducible highest weight module  as a submodule of the Fock space in affine type $A$. This gives a graded analogue of Ariki's Theorem, which interprets the usual decomposition matrix in terms of canonical bases. 

Actually, there exist several possible realizations of the irreducible highest  weight modules as submodules of the Fock space, and thus  several canonical bases matrices, in affine type $A$.  Therefore, it is natural to ask whether all these matrices can 
 be expressed as graded decomposition matrices. 
 From this perspective, one problem is to obtain a canonical way to define these graded decomposition matrices; indeed, Brundan and Kleshchev's approach depends on the existence and the choice of the Specht modules for the Hecke algebras.
 
 The aim of this note is to develop a general theory for graded decomposition matrices, using the concept of specialization.  This pursuit is motivated by the works above, 
  but can  be also interesting as a subject in its own right, providing a graded analogue of the  usual notion of decomposition matrices associated with specialization maps.
    In particular, we obtain a factorization result which is a graded analogue of  \cite[Proposition 2.6]{G}. Combined with the results obtained in \cite{AJL}, 
   the results of this note indicate the existence of several graded  representation theories for Hecke algebras of type $G(l,1,n)$ (see Section \ref{3.2}).

\section{Decomposition map for graded algebras}

The goal of this section is to define and study the graded decomposition map for graded algebras in a general setting.  
Here we adopt  the same approach as the one used in \cite{G} and  \cite[\S 7.3]{GP}, where the usual decomposition matrix associated with a specialization is defined. 
Hence, our definition is slightly different from the one presented in, for example, \cite{BK1,BK2,BS}. 

 \subsection{Graded algebras and representations}
 We begin with some notation.  Let $R$ be a commutative integral domain  with $1$ and
 let $H$ be an associative algebra, finitely generated and free over  $R$. 
 We say that $H$ is a \emph{$\mathbb{Z}$-graded algebra} if  $H$ admits $R$-submodules $(H_i)_ {i\in \mathbb{Z}}$  such that 
  \begin{itemize}
  \item $H=\bigoplus_{i\in \mathbb{Z}} H_i$, and
  \item $H_i\, H_j\subseteq H_{i+j}$ for all $(i,j)\in \mathbb{Z}^2$. 
  \end{itemize}
 Note that $H_0$ is a subalgebra of $H$. A \emph{graded $H$-module}  $M$ is an $H$-module equipped with a grading 
   \begin{itemize}
  \item $M=\bigoplus_{i\in \mathbb{Z}} M_i$ for $R$-submodules $M_i$ of $M$, and
  \item $H_i\, M_j\subseteq M_{i+j}$ for all $(i,j)\in \mathbb{Z}^2$. 
  \end{itemize}
  If $M$ is a graded $H$-module, we denote by $\underline{M}$ the (ungraded) $H$-module obtained from $M$ by forgetting the grading. 
  For $j\in \mathbb{Z}$, we denote by $M\langle j\rangle$ the graded module obtained by shifting $M$ up by $j$, that is, 
   $M\langle j\rangle_i:=M_{i-j}$ for all $i\in \mathbb{Z}$. 
  
Let $M$ and $N$ be graded $H$-modules.
A \emph{morphism of graded modules} $f:M\to N$
is a morphism of ungraded $H$-modules satisfying $f(M_i)\subseteq N_i$ for all $i\in \mathbb{Z}$.  If moreover $f$ is bijective, then $f$ is an \emph{isomorphism of graded modules}.

From now on, $R$ will be a field, and all modules (graded and ungraded ones) will be considered finite-dimensional. 

We denote by $\underline{\textrm{Irr}}(H)$ the set of simple $H$-modules. If $\underline{M} \in \underline{\textrm{Irr}}(H)$, then, by \cite[Lemma 2.5.3]{BGS}, $\underline{M}$ admits a unique $\mathbb{Z}$-grading up to shift. The associated modules $M\langle j\rangle$ are pairwise non-isomorphic simple graded $H$-modules. In addition, all simple graded $H$-modules are obtained in this way (see \cite[Theorem 9.6.8, Corollary  9.6.7]{NVO}). If $L$ is a graded $H$-module, we denote by $[L:M]^{\textrm{gr}}$ the multiplicity of the simple graded module $M$ in a graded composition series of $L$.

We denote by $\mathcal{R}_0 (H)$ the Grothendieck group of finite-dimensional  graded $H$-modules. It is well-known that the associated category  is an abelian category. 
We denote by   $\mathcal{R}_0^+ (H)$ the  subset of $\mathcal{R}_0 (H)$ consisting of the elements $[V]$,  where $V$ is a graded $H$-module. 
Let now $q$ be an indeterminate. If we fix one and for all a grading for each simple $H$-module $\underline{M}$, we obtain an action of  $\mathbb{Z}[q,q^{-1}]$  on  $\mathcal{R}_0 (H)$  given by
$q^j [M]:=M\langle j\rangle$, for all $j\in \mathbb{Z}.$
  As a  $\mathbb{Z}[q,q^{-1}]$-module,  $\mathcal{R}_0 (H)$  has a basis
    $\{ [M]\ |\ \underline{M} \in \underline{\textrm{Irr}}(H)\}.$

    Finally, let  $V=\oplus_{i\in \mathbb{Z}} V_i $ be a  graded $H$-module.
     Then, for each $i\in \mathbb{Z}$, $V_i$ is an $H_0$-module and we denote by $\chi_{V_i}$ the character of the associated representation ${\rho_{V_i}:H_0 \to  \operatorname{End}_R(V_i)}$. Let $t$ be an indeterminate. 
     We define the \emph{graded character}      of $V$ to be the $R$-linear map $\chi^{\textrm{gr}}_V : H_0 \to R[t,t^{-1}]$ such that  $\chi_V^{\textrm{gr}} (h_0)=\sum_{i\in \mathbb{Z}}  t^{i} \chi_{V_i} (h_0)$ for all $h_0\in H_0$.   
     We can extend  the definition of $\chi_V^{\textrm{gr}}$ to $H$ linearly, by setting $\chi_V^{\textrm{gr}} (h_k): =0$ for all $h_k \in H_k$
 with $k \neq0$.  Note that for  $j\in \mathbb{Z}$, we have $\chi^{\textrm{gr}}_{V\langle j\rangle}=t^j \chi^{\textrm{gr}}_{V}$. 
 Note also that, for all $h \in H$, the element 
       $\chi_{{V}}^{\textrm{gr}}(h)$ evaluated at $t=1$ is simply $\chi_{\underline{V}} (h)$, where $\chi_{\underline{V}}$ denotes the character of the representation of $H$ afforded by $\underline{V}$.
We denote by ${\textrm{Irr}} (H)$ the set of irreducible characters of $H$, and by  ${\textrm{Irr}}^{\textrm{gr}} (H)$ the set of irreducible graded characters of $H$. That is, 
$${\textrm{Irr}} (H) = \{ \chi_{\underline{M}}\ | \   \underline{M} \in \underline{\textrm{Irr}}(H)\},$$
and
 $${\textrm{Irr}}^{\textrm{gr}} (H) = \{ \chi^{\textrm{gr}}_{M\langle j\rangle}\ | \   \underline{M} \in \underline{\textrm{Irr}}(H),\, j\in \mathbb{Z}\}.$$
 We have ${\textrm{Irr}} (H) \subset \operatorname{Hom}_R (H,R)$  and ${\textrm{Irr}}^{\textrm{gr}} (H) \subset \textrm{Hom}_R (H,R[t,t^{-1}])$.

     \begin{lemma}\label{LI}
     Assume that the set  ${\operatorname{Irr}} (H)$ is a linearly independent subset of 
     $\operatorname{Hom}_R (H,R)$. Then ${\operatorname{Irr}}^{\operatorname{gr} }(H)$ is a linearly independent subset  of $\operatorname{Hom}_R (H,R[t,t^{-1}])$.
     \end{lemma}
     \begin{proof}
    Suppose that  
    $$\sum_{\underline{M}\in \underline{\textrm{Irr}} (H)} \sum_{j\in \mathbb{Z}} a_{\underline{M},j} \chi^{\textrm{gr}}_{M\langle j\rangle}=0$$
     for elements $a_{\underline{M},j}\in R$.
     Then we obtain:  
         $$\sum_{\underline{M}\in \underline{\textrm{Irr}} (H)}\sum_{j\in \mathbb{Z}} a_{\underline{M},j} t^j\chi^{\textrm{gr}}_{M}=\sum_{j\in \mathbb{Z}} t^j \sum_{\underline{M}\in \underline{\textrm{Irr}} (H)} a_{\underline{M},j} \chi^{\textrm{gr}}_{M}=0,$$   
whence we deduce  that, for all $j\in \mathbb{Z}$ and $h\in H$, we have 
         $$\sum_{\underline{M}\in \underline{\textrm{Irr}} (H)}  a_{\underline{M},j} \chi^{\textrm{gr}}_{M}(h)=0.$$
 Evaluating the above left-hand side element of $R[t,t^{-1}]$ at $t=1$ yields that        
              $$\sum_{\underline{M}\in \underline{\textrm{Irr}} (H)}  a_{\underline{M},j} \chi_{\underline{M}} (h)=0$$      
for all $h \in H$, that is,  
       $$\sum_{\underline{M}\in \underline{\textrm{Irr}} (H)} a_{\underline{M},j} \chi_{\underline{M}}=0.$$
       Since  ${\operatorname{Irr}} (H)$ is a linearly independent subset of 
     $\operatorname{Hom}_R (H,R)$, we conclude that $a_{\underline{M},j}=0$ for all $\underline{M}\in \underline{\textrm{Irr}} (H)$ and  $j \in \mathbb{Z}$.
     \end{proof}
   \begin{Rem}\label{firstremark}
   As noted in \cite[\S 7.3.3]{GP}, the assumption of Lemma \ref{LI} is satisfied when $H$ is split or $R$ is a perfect field. 
   \end{Rem}

 \subsection{A graded Brauer-Nesbitt Lemma}
 
 The definition of the graded decomposition map is based on a graded version of the Brauer-Nesbitt lemma, which we now give. 
Let $X$ be an indeterminate over the field  $R$  
and let $\maps (H_0,R[X]^{\mathbb{Z}})$ be the set of  maps from $H_0$ to $R[X]^{\mathbb{Z}}$.
 We define 
 $$\mathfrak{p}_R : \mathcal{R}_0^+ (H) \to \maps (H_0,R[X]^{\mathbb{Z}})$$
 such that, for a  graded $H$-module $V=\oplus_{i\in \mathbb{Z}} V_i$ and for $h_0 \in H_0$, 
 $$\mathfrak{p}_R ([V])(h_0)=\left(\textrm{Poly} (\rho_{V_i} (h_0))\right)_{i\in \mathbb{Z}},$$
  where $\textrm{Poly}$ denotes the characteristic polynomial and $\rho_{V_i}$ is the representation of $H_0$ afforded by $V_i$.
If we  consider  $\maps (H_0,R[X]^{\mathbb{Z}})$ as a semigroup with respect to pointwise multiplication of maps, then $\mathfrak{p}_R$ is a semigroup homomorphism.
Moreover, if $q$ is an indeterminate as before, we can define an action of  $\mathbb{N}[q,q^{-1}]$ on $\maps (H_0,R[X]^{\mathbb{Z}})$ as follows:
If $\psi \in \maps (H_0,R[X]^{\mathbb{Z}})$ and $h_0 \in H_0$, then $\left( q^j. \psi \right)(h_0):=
\left( (\psi (h_0) )_{i-j} \right)_{i\in \mathbb{Z}}$. One can easily check that $\mathfrak{p}_R$ preserves the action of $\mathbb{N}[q,q^{-1}]$  on $\mathcal{R}_0^+ (H)$. Thus, we will say that
$\mathfrak{p}_R$ is a $\mathbb{N}[q,q^{-1}]$-semigroup homomorphism.
The proof of the following lemma is inspired by \cite[Lemma 7.3.2]{GP}.

%
%

  \begin{lemma}
  Assume that ${\operatorname{Irr}}^{\operatorname{gr} }(H)$ is a linearly independent subset  of $\operatorname{Hom}_R (H,R[t,t^{-1}])$. Then $\mathfrak{p}_R$
  is injective.
  \end{lemma}
  \begin{proof}
  Let $V$ and $W$ be graded $H$-modules such that $\mathfrak{p}_R([V])=\mathfrak{p}_R ([W])$.  Since the classes $[V]$ and $[W]$ only depend on the graded composition factors of $V$ and $W$ respectively, without loss of generality, 
   we may assume that there exist $a_{j}$ and $b_{j}$ for  $j\in \mathbb{Z}$ such that 
   $$V= \bigoplus_{j\in\mathbb{Z}} a_{j} M_j  \qquad \textrm{and } \qquad W= \bigoplus_{j\in\mathbb{Z}} b_{j} M_j, $$
   where ${M}_j$ are pairwise non-isomorphic  simple graded $H$-modules. 
   If, for some $j \in \mathbb{Z}$, we have both $a_j >0$ and $b_j >0$, then we can write $V = M_j \oplus V' $ and $W = M_j \oplus W'$. Since 
   $\mathfrak{p}_R$ is a $\mathbb{N}[q,q^{-1}]$-semigroup homomorphism, we deduce that $\mathfrak{p}_R([V'])=\mathfrak{p}_R ([W'])$. Thus, without loss of generality,  we may also assume that, for all $j\in \mathbb{Z}$, we have either $a_{j}=0$ or $b_{j}=0$. 
   
   Now, by hypothesis, we have that,
    for all $h_0\in H_0$ and $i\in \mathbb{Z}$,
   $ \textrm{Poly} (\rho_{V_i} (h_0)) = \textrm{Poly} (\rho_{W_i} (h_0)).$
  As the character values appear in the  characteristic polynomial, this implies that $\chi_{V_i} (h_0)=\chi_{W_i} (h_0)$ for all $h_0\in H_0$ and $i\in \mathbb{Z}$. 
   We deduce that
   $$\chi_{V}^{\textrm{gr}}=\sum_{j\in \mathbb{Z}} a_j \chi_{M_j}^{\textrm{gr}}= \sum_{j\in \mathbb{Z}} b_j \chi_{M_j}^{\textrm{gr}}=\chi_{W}^{\textrm{gr}}.$$
   Since the set ${\operatorname{Irr}}^{\operatorname{gr} }(H)$ is a linearly independent subset  of $\operatorname{Hom}_R (H,R[t,t^{-1}])$,
 we must have,  for all $j\in \mathbb{Z}$, $(a_j-b_j) 1_R=0$. Since also $a_j=0$ or $b_j=0$, we must have $a_j 1_R=0$ and $b_j 1_R=0$ for all $j\in \mathbb{Z}$. 
   We can thus conclude in characteristic $0$. If the characteristic of $R$ is $p>0$, then we deduce that $p$ divides 
     all the $a_j$'s and $b_j$'s. We can then repeat the same argument by considering the elements $(1/p) [V]$ and $(1/p) [W]$ inside 
   $\mathcal{R}_0^+ (H)$ instead of $[V]$ and $[W]$. 
  \end{proof}
  
    \begin{Rem}\label{secondremark}
   Following Lemma \ref{LI} and Remark \ref{firstremark}, the assumption of the graded Brauer-Nesbitt Lemma is satisfied whenever $H$ is split or $R$ is a perfect field. 
   \end{Rem}

  We also obtain the following graded analogue of \cite[Lemma 7.3.4]{GP}.
    \begin{lemma}\label{whysplit}
     Let $R'$ be a field such that $R \subseteq R'$ and set $R'H:=R'\otimes_R H$. Then there exists a canonical map $d_R^{R'} : \mathcal{R}_0 (H) \to \mathcal{R}_0 (R'H)$  such that, 
     for a  graded $H$-module $V$, $d_R^{R'} ([V])=[R'\otimes_R V]$. Furthermore, we have a commutative diagram
  $$\def\commutatif{\ar@{}[rd]|{\circlearrowleft}}
  \xymatrix{ 
    \mathcal{R}_0^+ (H)    \ar[r] ^{\mathfrak{p}_R \ \ \ \ \ }\ar[d]^{d_{R}^{R'}\ \ \ \ \ }  \commutatif &  \maps (H_0,R[X]^{\mathbb{Z}})     \ar[d]^{t^{R'}_{R}} \\
    \mathcal{R}_0^+ ( R' H) \ar[r]^{\mathfrak{p}_{R'} \ \ \ \ \ } &  \maps (H_0,R'  [X]^{\mathbb{Z}})  }$$
    where $t^{R'}_R$ is the canonical embedding. If, in addition, $H$ is split,  then $d_R^{R'}$ is an isomorphism which preserves isomorphism classes of simple graded  modules. 
  \end{lemma}
%
%
  
  \subsection{Modular reduction for graded modules}\label{modular}
  
  From now on, we will work in the following setting:
   Let $A$ be an integral domain  and let $H$ be an associative $A$-algebra, finitely generated
  and free over $A$ 
   such that $H$ is $\mathbb{Z}$-graded:
  $$H=\bigoplus_{i\in \mathbb{Z}} H_i.$$
Let $K$ be the field of fractions of $A$ and set $KH:=K\otimes_A H$.   Note that 
 $KH$ is naturally equipped with a $\mathbb{Z}$-grading which comes form the grading of $H$: 
   $$KH=\bigoplus_{i\in \mathbb{Z}} KH_i.$$
Let $\mathcal{O}$ be a valuation ring such that $A \subseteq \mathcal{O} \subseteq K$. Then the following lemma holds.
   
      \begin{lemma}\label{lemmamo}
Let  $V$ be a finite-dimensional graded $KH$-module. 
 There exists a graded  $\mathcal{O}H$-lattice  $\widetilde{V}$ satisfying 
$$
V\cong K\otimes_{\mathcal{O}} \widetilde{V}.
$$
   \end{lemma}
   \begin{proof}
    Assume that $V=\oplus_{j\in \mathbb{Z}} V_j$ is the fixed grading of $V$. For $j\in \mathbb{Z}$, let
     $\{v_1^j,v_2^j,\ldots,v_{r_j}^j\}$ be a $K$-basis of $V_j$. Moreover, for $i \in \mathbb{Z}$, let $\{b_1^i,b_2^i,\ldots,b_{s_i}^i\}$ be an $A$-basis
      of $H_i$.  Then, for  $\ell \in \mathbb{Z}$, let $\widetilde{V}_\ell$ be the $\mathcal{O}$-submodule of $V_\ell$ spanned by the finite set 
   $$\{b_m^i v_n^j\ |\ 1\leq m \leq s_i,\ 1\leq n \leq r_j,\ i+j=\ell\}.$$
     Set $\widetilde{V}:=\bigoplus_{\ell \in \mathbb{Z}}\widetilde{V}_\ell$. 
     Then 
   $ \widetilde{V}$ is a finitely generated graded $\mathcal{O}H$-module.  Since it is contained in a $K$-vector space, it is also torsion-free.
   It is a well-known fact (see, for example, \cite[7.3.5]{GP}) that every finitely generated torsion-free module over a valuation ring is free.
   Thus,  $ \widetilde{V}$ is a graded $\mathcal{O}H$-lattice satisfying $V\cong K\otimes_{\mathcal{O}} \widetilde{V}$,
    as any $\mathcal{O}$-basis of 
      $\widetilde{V}$ is also a  $K$-basis of ${V}$.      
%
%
     \end{proof}
   
   Let $V=\oplus_{j\in \mathbb{Z}} V_j$ and $\widetilde{V}=\bigoplus_{\ell \in \mathbb{Z}}\widetilde{V}_\ell$ be as in the lemma above.
   Let $i \in \mathbb{Z}$. By construction of $\widetilde{V}_i$, there exists a $K$-basis of $V_i$ such that the corresponding matrix representation $\rho_{V_i}: KH_0 \to M_{n_i}(K)$ (where $n_i=\mathrm{dim}_KV_i$) has the property that $\rho_{V_i}(h_0) \in M_{n_i}(\mathcal{O})$ for all $h_0 \in H_0$. 
   We deduce that $\mathfrak{p}_{K}([V])(h_0) \in \mathcal{O}[X]^{\mathbb{Z}}$ for all $h_0 \in H_0$.
   
   The following result is a direct application of the above construction. It uses the fact that the integral closure of $A$ in $K$ is equal to the intersection of all valuation rings $\mathcal{O}$ such that $A \subseteq \mathcal{O} \subseteq K$. For the same result in the ungraded setting, see \cite[Proposition 7.3.8]{GP}.
   
   \begin{Prop}\label{intclos}
   Let $V$ be a finite-dimensional graded $KH$-module and let ${A}^*$ be the integral closure of $A$ in $K$.  Then we have $\mathfrak{p}_{K}([V])(h_0) \in 
   A^* [X]^{\mathbb{Z}}$ for all $h_0 \in H_0$. Thus, the map $\mathfrak{p}_{K}$  is in fact a map $\mathfrak{p}_{K}:\mathcal{R}_0^+ (KH) \to \maps (H_0,A^*[X]^{\mathbb{Z}})$.
   \end{Prop}
   
From now on, we will assume that $A$ is integrally closed in $K$, that is, $A^*=A$.
   Let $\theta : A \to L$ be a  ring homomorphism into a field $L$ such that $L$ is the field of fractions of $\theta (A)$. We obtain an $L$-algebra
    $LH:=L\otimes_A H$, where $L$ is regarded as an $A$-module via $\theta$. We have a canonical grading: 
   $$LH=\bigoplus_{i\in \mathbb{Z}} L H_i.$$


 
Let $\mathcal{O} \subseteq K$ be a valuation ring with  maximal ideal $J(\mathcal{O})$ such that $A\subseteq \mathcal{O}$ and 
   $J(\mathcal{O})\cap A =\textrm{Ker} (\theta)$. The existence of such a ring is guaranteed by a standard result on valuation rings (see, for example, \cite[7.3.5]{GP}). 
  Let $k$ be the residue field of $\mathcal{O}$ and set $\mathcal{O}H:=\mathcal{O}\otimes_A H$ and ${k}H:={k}\otimes_{\mathcal{O}} H$. 
  The restriction of the canonical map $\pi: \mathcal{O} \twoheadrightarrow k$ to $A$ has kernel  $J(\mathcal{O})\cap A =\textrm{Ker} (\theta)$. Since $L$ is the field of fractions of $\theta(A)$,
$L$ can be regarded as a subfield of $k$. One can now identify the graded Grothendieck groups of 
 $LH$ and $kH$ if the following assumption holds (recall that if $LH$ is split, then the map $d_L^k$ of Lemma \ref{whysplit} is an isomorphism which preserves isomorphism classes of simple  graded  modules):

\begin{assump}\label{hyp}
Let $L$ and $k$ be as above. We have that
 \begin{enumerate}[(a)]
\item $LH$ is split, or
\item $L=k$ and $L$ is perfect.
\end{enumerate}
\end{assump}
In particular, following Remark \ref{secondremark}, under the above assumption the graded Brauer-Nesbitt Lemma holds and the map $\mathfrak{p}_L$ is injective.

Now let $V= \bigoplus_{i\in \mathbb{Z}}{V}_i$ be any graded $KH$-module.  By Lemma \ref{lemmamo}, there exists a graded  $\mathcal{O}H$-lattice  $\widetilde{V}=\bigoplus_{i\in \mathbb{Z}}\widetilde{V}_i$ such that
$V\cong K\otimes_{\mathcal{O}} \widetilde{V}.$ The $k$-vector space $k \otimes_{\mathcal{O}} \widetilde{V}$ is a graded $kH$-module via $(\widetilde{v} \otimes 1)(h \otimes 1)=\widetilde{v}h \otimes 1\,(\widetilde{v} \in \widetilde{V},\, h \in H)$, which we call the \emph{modular reduction} of $\widetilde{V}$.  To simplify notation, we shall write
$$K \widetilde{V}:=K\otimes_{\mathcal{O}} \widetilde{V} \quad\text{and}\quad k \widetilde{V}:=k\otimes_{\mathcal{O}} \widetilde{V}.$$

Let $i \in \mathbb{Z}$. Following the remarks after Lemma \ref{lemmamo}, there exists a $K$-basis of $V_i$ such that the corresponding matrix representation $\rho_{V_i}: KH_0 \to M_{n_i}(K)$ (where $n_i=\mathrm{dim}_KV_i$) has the property that $\rho_{V_i}(h_0) \in M_{n_i}(\mathcal{O})$ for all $h_0 \in H_0$. The matrix representation $\rho_{k\widetilde{V_i}}: kH_0 \to M_{n_i}(k)$ afforded by $k\widetilde{V_i}$ is given by 
\begin{equation}\label{repk}\tag{$\star$}
\rho_{k\widetilde{V_i}}(h_0 \otimes 1)=\pi \left(\rho_{V_i}(h_0)\right).
\end{equation}

   \subsection{Graded decomposition map}

The aim of this subsection is to relate the graded representation theory of $LH$ with that of $KH$, in the setting of \S \ref{modular}. Recall that we have fix a grading for the simple $KH$ and $LH$-modules once for all. 
We will define a decomposition map
  $$d^{\operatorname{gr}}_{\theta}: \mathcal{R}_0 (KH) \to \mathcal{R}_0 (L H)$$
  between the associated Gorthendieck groups of finite-dimensional  graded modules.
 
 \begin{Th}\label{uni}
 Assume that $A$ is integrally closed in $K$ and let  $\theta : A \to L$ be a  ring homomorphism into a field $L$ such that $L$ is the field of fractions of $\theta (A)$.
Let $\mathcal{O} \subseteq K$ be a valuation ring with  maximal ideal $J(\mathcal{O})$ such that $A\subseteq \mathcal{O}$ and 
   $J(\mathcal{O})\cap A =\textrm{Ker} (\theta)$.  If Assumption \ref{hyp} is satisfied, then
 the following assertions hold. 
 \begin{enumerate}[(1)]
%
%

  \item    The modular reduction induces a  $\mathbb{N}[q,q^{-1}]$-semigroup homomorphism
       $$d^{\operatorname{gr}}_{\theta}: \mathcal{R}^+_0 (K H) \to  \mathcal{R}^+_0 (LH)$$
           by setting 
    $$\displaystyle d^{\operatorname{gr}}_{\theta}([K\widetilde{V}]):=  [k \widetilde{V}],     $$
    where     $\widetilde{V}$ is a graded $\mathcal{O}H$-lattice and $[k  \widetilde{V}]$ is regarded as an element of 
     $\mathcal{R}_0^+ (LH)$ via the identification with $\mathcal{R}_0^+ (kH)$. 
 \item By Proposition \ref{intclos}, the image of $\mathfrak{p}_K$ is contained in $\maps (H_0,A[X]^{\mathbb{Z}})$. We have a commutative diagram 
  $$\def\commutatif{\ar@{}[rd]|{\circlearrowleft}}
  \xymatrix{ 
    \mathcal{R}_0^+ (KH)    \ar[r] ^{\mathfrak{p}_K \ \ \ \ \ }\ar[d]^{d_{\theta}^{\textrm{gr}}\ \ \ \ \ }  \commutatif &  \maps (H_0,A[X]^{\mathbb{Z}})     \ar[d]^{t_{\theta}} \\
    \mathcal{R}_0^+ (L  H) \ar[r]^{\mathfrak{p}_{L } \ \ \ \ \ } &  \maps (H_0,L  [X]^{\mathbb{Z}})  }$$

 where $t_{\theta}: \maps (H_0,A[X]^{\mathbb{Z}}) \to \maps (H_0, L  [X]^{\mathbb{Z}})$ is the natural map obtained by applying $\theta$ to the coefficients of the polynomials in $A[X]$. 
 \item The map $d^{\operatorname{gr}}_{\theta}$ is the unique map satisfying the commutative diagram in (2); in particular, $d^{\operatorname{gr}}_{\theta}$ does not depend on the choice of $\mathcal{O}$.

 \end{enumerate}

 \end{Th} 
 \begin{proof}
For (1), we need to show that $\displaystyle d^{\operatorname{gr}}_{\theta}$ is well-defined. Thus, 
 we need to show that if $\widetilde{V}$ and $\widetilde{W}$ are graded $\mathcal{O}H$-lattices such that
 $[K\widetilde{V}]=[K\widetilde{W}]$, then  $[k\widetilde{V}] = [k\widetilde{W}] $.
Note that Equation (\ref{repk}) implies that 
$$p_L ([k\widetilde{V}])=t_{\theta} ( p_K ([K\widetilde{V}]))=t_{\theta} ( p_K ([K\widetilde{W}]))=p_L ([k\widetilde{W}]).$$ 
As mentioned in the previous subsection, under Assumption \ref{hyp} the graded Brauer-Nesbitt Lemma holds and the map   $p_L$ is injective. Hence, we deduce that $[k\widetilde{V}] = [k\widetilde{W}] $, as desired.
Thus,  $\displaystyle d^{\operatorname{gr}}_{\theta}$ is well-defined and  (2) is also proved. Finally, since $p_L$ is injective, there is at most one map 
     $\displaystyle d^{\operatorname{gr}}_{\theta}$ that makes the diagram in (2) commutative. This proves (3). 
 \end{proof}
 
 The map $d^{\operatorname{gr}}_{\theta}$ will be called the \emph{graded decomposition map} associated with the specialization $\theta :A \to L$. We can now easily extend its
  definition to the Grothendieck groups $\mathcal{R}_0 (KH)$ and $\mathcal{R}_0 (LH)$. 

 Let us consider a simple $KH$-module $\underline{V}$. Then $V$ is graded.  Let $\widetilde{V}$ be a graded  $\mathcal{O}H$-lattice such that $V \cong K\widetilde{V}$. Then we have
  $$\begin{array}{rclll}
     d^{\operatorname{gr}}_{\theta } ([K\widetilde{V}])  &=&        [k\widetilde{V}] & =&
 \displaystyle \sum_{\underline{M}\in \underline{\textrm{Irr}} (L  H)} \sum_{n\in \mathbb{Z}} [k\widetilde{V} :M\langle n\rangle]^{\textrm{gr}} [M\langle n\rangle] \\
 & &  & =&\displaystyle  \sum_{\underline{M}\in \underline{\textrm{Irr}} (L  H)}   \left( \sum_{n\in \mathbb{Z}} q^n  [k\widetilde{V} :M\langle n\rangle]^{\textrm{gr}} \right) [M]
   \end{array}
   $$ 
 We can then define a \emph{graded decomposition matrix} with coefficients in $\mathbb{N}[q,q^{-1}]$ as follows:
 $$D^{\operatorname{gr}}_{\theta} (q):= \left(  \sum_{n\in \mathbb{Z}} q^n  [k\widetilde{V}:M\langle n\rangle]^{\textrm{gr}}  \right)_{\underline{V}\in \underline{\textrm{Irr}} (KH),\, \underline{M} \in \underline{\textrm{Irr}} (LH)}.$$
 One can see that $D^{\operatorname{gr}}_{\theta} (1)$ is the usual decomposition matrix  for $H$  associated with the specialization $\theta$. 
 
    
  \subsection{Factorization property of the graded decomposition map}
  
 We wish to describe a factorization result which is  a graded analogue of \cite[Proposition 2.6]{G}.
  Assume that $A$ is an integrally closed ring.
  Let $\theta : A \to L$ and $\theta' : A \to L'$ be specializations such that $L$ is the field of fractions of $\theta (A)$ and 
   $L'$ is the field of fractions of $\theta ' (A)$.  Let us assume that 
   $$\textrm{Ker} (\theta) \subseteq \textrm{Ker} (\theta ').$$
   Then we can define a ring homomorphism $\phi: \theta (A) \to L'$ such that $\phi \, (\theta(a))=\theta'(a)$ for all $a \in A$, that is, $\phi \circ \theta=\theta'$.  
   Note that 
  $ \phi \, (\theta(A))=\theta ' (A)$ and thus $L'$ is the field of fractions of $\phi \, (\theta(A))$. We assume in addition that $\phi$ can be extended to a  ring $B$ that is integrally closed in $L$ (which happens, for example, when $\theta (A)$ is integrally closed in $L$ and $B=\theta (A)$). 
  
 Now assume that the algebras $LH$ and $L'H$ are split. Using our previous results, we can define three decomposition maps:
     $$d^{\operatorname{gr}}_{\theta}: \mathcal{R}^+_0 (KH) \to \mathcal{R}^+_0 (L H),\qquad d^{\operatorname{gr}}_{\theta'}: \mathcal{R}^+_0 (KH) \to \mathcal{R}^+_0 (L' H), $$
  $$d^{\operatorname{gr}}_{\theta,\theta '}: \mathcal{R}^+_0 (LH) \to \mathcal{R}^+_0 (L' H),$$
      where the last map is defined with respect to $t_{\theta,\theta'}:\, \maps ( H_0,B [X]^{\mathbb{Z}})\to   \maps (H_0,L' [X]^{\mathbb{Z}})$, which is the map obtained by applying $\phi$ to the coefficients of the polynomials in $B [X]$.
 \begin{Prop}\label{facto} The following diagram is commutative:
$$
 \xymatrix @R=1cm @C=1cm{ \mathcal{R}^+_0 (KH)  \ar[rr]^{d^{\textrm{gr}}_{{\theta '}}  } \ar[rd]^{d^{\textrm{gr}}_{{\theta }}} && \mathcal{R}^+_0(L 'H ) \\ & \mathcal{R}^+_0 (L {H}  )\ar[ru]^{d^{\textrm{gr}}_{{\theta,\theta '}}  }  }
$$
In particular, if $D_{\theta}^{\textrm{gr}} (q)$, $D_{\theta'}^{\textrm{gr}} (q)$   and  $D_{\theta,\theta'}^{\textrm{gr}} (q)$ are the associated graded decomposition matrices,  then
$$D_{\theta'}^{\textrm{gr}} (q)=D_{\theta,\theta'}^{\textrm{gr}}  (q)\, D_{\theta}^{\textrm{gr}} (q).$$
 \end{Prop}
 \begin{proof}
 
 We have the following commutative diagrams (note that $t_{\theta}$ takes its values in  $\maps (H_0,B  [X]^{\mathbb{Z}})$):
 
   $$\def\commutatif{\ar@{}[rd]|{\circlearrowleft}}
  \xymatrix{ 
    \mathcal{R}_0^+ (KH)    \ar[r] ^{\mathfrak{p}_K \ \ \ \ \ }\ar[d]^{d_{\theta}^{\textrm{gr}}\ \ \ \ \ }  \commutatif &  \maps (H_0,A[X]^{\mathbb{Z}})   \ar[d]^{t_{\theta}  }     \\
    \mathcal{R}_0^+ (L  H) \ar[r]^{\mathfrak{p}_{L } \ \ \ \ \ }\ar[d]^{d_{\theta,\theta'}^{\textrm{gr}}\ \ \ \ \ }  \commutatif  &  \maps (H_0,B  [X]^{\mathbb{Z}}) \ar[d]^{t_{\theta,\theta'} }   \\
   \mathcal{R}_0^+ (L'  H) \ar[r]^{\mathfrak{p}_{L' } \ \ \ \ \ } &  \maps (H_0,L'  [X]^{\mathbb{Z}})  
   }
    $$
We have $t_{\theta'}=t_{\theta,\theta'}\circ t_{\theta}$. Thus, we
     obtain a commutative diagram:
    $$\def\commutatif{\ar@{}[rd]|{\circlearrowleft}}
  \xymatrix{ 
    \mathcal{R}_0^+ (KH)    \ar[r] ^{\mathfrak{p}_K \ \ \ \ \ }\ar[d]^{d_{\theta,\theta'}^{\textrm{gr}} \circ \, d_{\theta}^{\textrm{gr}}  \ \ \ \ \ }  \commutatif &  \maps (H_0,A[X]^{\mathbb{Z}})     \ar[d]^{t_{\theta'}} \\
    \mathcal{R}_0^+ (L'  H) \ar[r]^{\mathfrak{p}_{L' } \ \ \ \ \ } &  \maps (H_0,L'  [X]^{\mathbb{Z}})  }$$
By Theorem \ref{uni} $(3)$, we conclude that $d_{\theta'}^{\textrm{gr}}=d^{\textrm{gr}}_{{\theta,\theta '}} \circ d_{\theta}^{\textrm{gr}}$.
 \end{proof}

  \section{Open problems on Hecke algebras}\label{3.2}

The above results induce interesting questions on the representation theory of Hecke algebras.  Related problems have already been suggested in the introduction of \cite{AJL}. We
will discuss  here another point of view.

Let $l,\,n\in \mathbb{N}$. Let $R$ be a ring and let $(v_1,\ldots,v_l,v)$ be a collection of invertible elements in $R$.   Let $\mathcal{H}_{R}$ be the associated  Ariki-Koike algebra over $R$ with
\begin{itemize}
\item generators :  $T_0$, $T_1$, \ldots, $T_{n-1}$.
\item relations : $(T_0-v_1)\cdots (T_0-v_l)=
 (T_i-v)(T_i+1)=0$ for all $i=1,\ldots,n-1$, together with the braid relations of type $B_n$. 
\end{itemize}

Assume first that $R=\mathbb{C}$, $v=\textrm{exp}(2\pi i/e)$ and $v_j=v^{s_j}$ for all $j=1,\ldots,l$ for some $e>1$ and some ${\bs}=(s_1,\ldots,s_l)\in \mathbb{Z}^l$. 
Note that  the associated Ariki-Koike algebra $\mathcal{H}_{\mathbb{C}}$ depends only on $e$ and on the class of ${\bs}$ modulo the natural action of the affine symmetric group $\widehat{\mathfrak{S}}_l$ (see, for example, \cite[\S 6.2.9]{GJ}). We denote by $\widetilde{\bs}$ the class of $\bs$ modulo this action. 
In \cite{BK1, BK2} Brundan and Kleshchev have shown that 
  the algebra $\mathcal{H}_{\mathbb{C}}$   can be endowed with a $\mathbb{Z}$-grading. Using a grading on the usual Specht modules \cite{BKW} (see also \cite{Huma}), 
  a   graded decomposition matrix $D_{\widetilde{\bs}} (q)$ can  be then  defined by considering the graded composition series of these modules.
  
    On the other hand, let us consider the quantum affine algebra $\mathcal{U}_q (\widehat{\mathfrak{sl}}_e)$ of type $A$. There is an action of $\mathcal{U}_q (\widehat{\mathfrak{sl}}_e)$  on the Fock space $\mathcal{F}_{\bs}$ parametrized by $\bs=(s_1,\ldots, s_l) \in \mathbb{Z}^l$ which allows an explicit realization of the simple highest weight modules. The associated canonical bases (in the sense of Kashiwara-Lusztig) can be then encoded in a matrix $\mathcal{D}_{\bs} (q)$ with coefficients in $\mathbb{N}[q,q^{-1}]$. More importantly, in general we have 
  $\mathcal{D}_{\bs} (q)\neq \mathcal{D}_{\bs '} (q)$   if $\bs\neq \bs'$, even when $\bs$ and $\bs'$ are in the same class modulo the action of the affine symmetric group $\widehat{\mathfrak{S}}_l$.
  
Let us fix  a class $\widetilde{\bs}$ of $\mathbb{Z}^l$ modulo   $\widehat{\mathfrak{S}}_l$. The main result of Brundan and Kleshchev   asserts that for a special choice of an element $\bs$ in the class
$\widetilde{\bs}$ (corresponding to the case where the Fock space may be seen as a tensor product of level one Fock spaces), 
  $\mathcal{D}_\bs (q)$  coincides with the graded decomposition matrix   $D_{\widetilde{\bs}} (q)$. So it is natural to  ask what happens with the other matrices 
    $\mathcal{D}_{\bs'} (q)$ ,  where $\bs' \in \widetilde{\bs}$ (that is, $\bs' \equiv \bs\ (\textrm{mod } \widehat{\mathfrak{S}}_l)$) and $\bs'\neq \bs$. 
  
   Let us consider this problem using the concept of specialization. 
 A factorization result  for the matrices $\mathcal{D}_{\bs} (q)$ has been obtained in \cite{AJL} for arbitrary $\bs$. Set now $A:=\mathbb{C}[v^{\pm 1},v_1^{\pm 1},\ldots,v_l^{\pm 1}]$, where $v$, $v_1$, \ldots, $v_l$  are indeterminates, and assume that there exists a $\mathbb{Z}$-grading on the Ariki-Koike
algebra $\mathcal{H}_{A}$. If we consider the specialization $\theta : A \to \mathbb{C}$ such that 
 $\theta (v)=\textrm{exp}(2\pi i/e)$ and $\theta (v_j)=\textrm{exp}(2\pi is_j/e)$, the results of this note imply the existence of a ``canonical'' graded decomposition matrix. Moreover, this matrix verifies a factorization result (Proposition \ref{facto}), which  may be viewed as an analogue of the one obtained  in \cite{AJL} for the matrices $\mathcal{D}_{\bs} (q)$.  This in turn suggests the existence of {several gradings} on 
$\mathcal{H}_{A}$, each yielding a canonical graded decomposition matrix which corresponds to one of the $\mathcal{D}_\bs (q)$'s. 

\vspace{1cm}
\noindent {\bf Addresses}\\

\noindent \textsc{Maria Chlouveraki}\footnote{The research project is implemented within the framework of the Action ``Supporting Postdoctoral Researchers'' of the Operational Program
``Education and Lifelong Learning'' (Action's Beneficiary: General Secretariat for Research and Technology), and is co-financed by the European Social Fund (ESF) and the Greek State.}, Laboratoire de Math\'ematiques,
UVSQ,
B\^atiment Fermat,
45 Avenue des Etats-Unis,
78035 Versailles  FRANCE\\ \emph{maria.chlouveraki@ed.ac.uk}\\

\noindent \textsc{Nicolas Jacon}, UFR Sciences et Techniques,
16 Route de Gray,
25030 Besan\c con
FRANCE\\  \emph{njacon@univ-fcomte.fr}


\begin{thebibliography}{99}                                                                                               %



\bibitem {AJL}\textsc{Ariki, S.,  Jacon, N. and Lecouvey, C.} Factorization of
the canonical bases for higher level Fock spaces, to appear in Proc. Eding.
Math. Soc, arXiv: 0909.2954.

\bibitem {BGS}
\textsc{Beilinson, A.,  Ginzburg, V. and  Soergel, W.} Koszul  duality patterns in representation theory. J. Amer. Math. Soc. 9 (1996), no. 2, 473--527.

\bibitem{BK1}
\textsc{Brundan, J. and  Kleshchev, A.} Blocks of cyclotomic Hecke algebras and Khovanov-Lauda algebras. Invent. Math. 178 (2009), no. 3, 451--484. 

\bibitem{BK2}
\textsc{Brundan, J. and  Kleshchev, A.} Graded decomposition numbers for cyclotomic Hecke algebras. Adv. Math. 222 (2009), no. 6, 1883--1942.

\bibitem{BKW}
\textsc{Brundan, J.,   Kleshchev, A. and Wang, W.} Graded Specht Modules.
J. Reine und Angew. Math. 655 (2011), 61--87. 

\bibitem{BS}
\textsc{Brundan, J. and Stroppel,  C.}
Highest weight categories arising from Khovanov's diagram algebra I: cellularity, to appear in Mosc. Math. J..

\bibitem{G}
\textsc{Geck, M.} 
Representations of Hecke algebras at roots of unity. Séminaire Bourbaki. Vol. 1997/98. Astérisque No. 252 (1998), Exp. No. 836, 3, 33--55.

\bibitem{GJ}
\textsc{Geck, M. and  Jacon, N.}
 Representations of Hecke algebras at roots of unity. Algebra and Applications, 15. Springer-Verlag London, Ltd., London, 2011. 



\bibitem{GP}
\textsc{Geck, M. and  Pfeiffer, G.}
 Characters of finite Coxeter groups and Iwahori-Hecke algebras. London Mathematical Society Monographs. New Series, 21. The Clarendon Press, Oxford University Press, New York, 2000. xvi+446 pp. 




\bibitem{Huma}
\textsc{Hu, J. and  Mathas, A.} 
Graded cellular bases for the cyclotomic Khovanov-Lauda-Rouquier algebras of type A. Adv. Math. 225 (2010), no. 2, 598--642.




\bibitem{KL}
\textsc{Khovanov, M. and  Lauda, A.} 
A diagrammatic approach to categorification of quantum groups. I. 
Represent. Theory 13 (2009), 309--347.





\bibitem{NVO}
\textsc{Nastasescu, C. and Van Oystaeyen, F.}
 Methods of graded rings. Lecture Notes in Mathematics, 1836. Springer-Verlag, Berlin, 2004. xiv+304 pp.

\bibitem{Rr}
\textsc{Rouquier, R.}
$2$-Kac- Moody algebras, preprint, arXiv:0812.5023.

\end{thebibliography}
\end{document}